\documentclass[12pt,a4paper]{article}
\usepackage{amsfonts}
\usepackage{amssymb}
\usepackage{mathrsfs}
\usepackage{amsmath}
\usepackage{amsthm}
\usepackage{enumerate}
\usepackage{cite}
\usepackage{xypic}
\usepackage{amscd}
\allowdisplaybreaks
\newtheorem{lemma}{Lemma}[section]
\newtheorem{theorem}[lemma]{Theorem}

\newtheorem{proposition}[lemma]{Proposition}
\newtheorem{definition}[lemma]{Definition}
\newtheorem{corollary}[lemma]{Corollary}

\begin{document}

\title{\textbf{Restricted hom-Lie algebras}
\author{ Baoling Guan$^{1,2}$,  Liangyun Chen$^{1}$
 \date{{\small {$^1$ School of Mathematics and Statistics, Northeast Normal
 University,\\
Changchun 130024, China}\\{\small {$^2$ College of Sciences, Qiqihar
University, Qiqihar 161006, China}}}}}}

\date{ }
\maketitle
\begin{quotation}
\small\noindent \textbf{Abstract}: The paper studies the structure
of restricted hom-Lie algebras. More specifically speaking, we first
give the equivalent definition of restricted hom-Lie algebras. Second, we obtain some
properties of $p$-mappings and restrictable hom-Lie algebras.
Finally, the cohomology of restricted hom-Lie algebras is researched.


\renewcommand{\thefootnote}{\fnsymbol{footnote}}
 \footnote[0]{Corresponding
author(L. Chen): chenly640@nenu.edu.cn.}
 \footnote[0]{Supported by NNSF
of China (No.11171055),  Natural Science Foundation of Jilin
province (No. 201115006), Scientific Research Foundation for
Returned Scholars
    Ministry of Education of China and the Fundamental Research Funds for the Central Universities(No. 11SSXT146).}
\end{quotation}
\setcounter{section}{0}

\section{Introduction}

    The concept of a  restricted Lie algebra is
       attributable to N. Jacobson  in 1943. It is well known that  the Lie algebras
       associated with algebraic groups   over a field of  characteristic $p$  are
       restricted Lie algebras \cite{sf}. Now,
   restricted theories attract more and more attentions.
 For example:  restricted Lie superalgebras\cite{pe},  restricted
       Lie color algebras\cite{ba}, restricted
        Leibniz  algebras\cite{dl},  restricted Lie triple systems\cite{htl} and
     restricted Lie algebras \cite{etj} were studied, respectively.

   However, The notion of hom-Lie algebras was introduced by Hartwig, Larsson and Silvestrov in \cite{hls}
 as part of a study of deformations of the Witt and the Virasoro algebras. In
a hom-Lie algebra, the Jacobi identity is twisted by a linear map, called the hom-Jacobi identity.
Some q-deformations of the Witt and the Virasoro algebras have the
structure of a hom-Lie algebra \cite{hls}. Because of close relation to discrete and deformed
vector fields and differential calculus \cite{hls,ld,ldd}, hom-Lie algebras are widely studied
recently \cite{af,bs,ma,mas,ydd,yddd,dy}.
   As   a natural generalization of a restricted Lie  algebra, it seems desirable to investigate the possibility of establishing
 a parallel theory  for  restricted  hom-Lie algebras.  As is well
known,  restricted Lie  algebras play predominant roles in the
theories of modular Lie algebras \cite{sf1}. Analogously,  the study
of restricted hom-Lie algebras  will play
      an important role in
    the  classification  of the  finite-dimensional modular simple hom-Lie algebras.

 The paper study the structure of
restricted hom-Lie algebras. Let us briefly
describe the content and  setup of the present article. In Sec. 2,
the equivalent definition of restricted hom-Lie algebras is given.
   In Sec. 3, we obtain some properties of $p$-mappings and restrictable hom-Lie algebras.
  In Sec. 4, we research the cohomology of restricted hom-Lie algebras.

 In the paper, $\mathbb{F}$ is a field of prime characteristic. Let $L$ denote a finite-dimensional restricted hom-Lie
 algebra over $\mathbb{F}.$

\begin{definition}{\rm \cite{sf}}\,  Let $L$ be a Lie algebra over $\mathbb{F}.$ A mapping
$[p]:L\rightarrow L, a\mapsto a^{[p]}$ is called a $p$-mapping, if

$\mathrm{(1)}$  $\mathrm{ad}{a^{[p]}}=(\mathrm{ad}{a})^{p},\ \forall a\in L,$

$\mathrm{(2)}$  $(\alpha a)^{[p]}=k^{p}a^{[p]},\ \forall a\in L, k\in \mathbb{F},$

$\mathrm{(3)}$
$(a+b)^{[p]}=a^{[p]}+b^{[p]}+\sum\limits_{i=1}^{p-1}s_{i}(a,b),$\\
where $ (\mathrm{ad}{(a\otimes X+b\otimes 1)})^{p-1}(a\otimes 1)=\sum\limits_{i=1}^{p-1}is_{i}(a,b)\otimes X^{i-1} $ in $L\otimes_{\mathbb{F}}\mathbb{F}[X], \forall a,b\in L,$
The pair $(L,[p])$ is referred to as a restricted Lie algebra.
\end{definition}

\begin{definition} {\rm\cite{syh}}\, $\mathrm{(1)}$ A hom-Lie algebra is a triple $(L, [.,.]_{L}, \alpha)$
consisting of a linear space $L,$ a skew-symmetric bilinear map $[.,.]_{L}: \Lambda^{2}L\rightarrow L$ and a linear map
$\alpha: L\rightarrow L$ satisfying the following hom-Jacobi identity:
$$[\alpha(x), [y, z]_{L}]_{L}+[\alpha(y), [z, x]_{L}]_{L}+[\alpha(z), [x, y]_{L}]_{L}=0$$
for all $x, y, z\in L;$

$\mathrm{(2)}$ A hom-Lie algebra is called a multiplicative hom-Lie algebra if $\alpha$ is an algebraic morphism, i.e., for any $x,y\in L,$
we have $\alpha([x,y]_{L})=[\alpha(x),\alpha(y)]_{L};$

$\mathrm{(3)}$ A sub-vector space $\eta\subset L$ is called a  hom-Lie subalgebra of $(L,[.,.]_{L},\alpha)$ if $\alpha(\eta)\subset \eta$ and $\eta$ is closed under
the bracket operation $[.,.]_{L},$ i.e., $[x,y]_{L}\in \eta$ for all $x,y\in \eta;$

$\mathrm{(4)}$ A sub-vector space $\eta\subset L$ is called a  hom-Lie ideal of $(L,[.,.]_{L},\alpha)$ if $\alpha(\eta)\subset \eta$ and $[x,y]_{L}\in \eta$ for all $x\in \eta, y\in L.$
\end{definition}

\section{The equivalent definition of restricted hom-Lie algebras}

 Let $(L, [,]_{L},\alpha)$ be a multiplicative hom-Lie algebra over $\mathbb{F}.$ For $c\in L$ satisfying $\alpha(c)=c,$ we define $\mathrm{ad}c(a):=[\alpha(a), c].$ Put $L_{0}:=\{x|\alpha(x)\neq x\}\cup\{0\}$ and $L_{1}:=\{x|\alpha(x)=x\}.$ Then $L=L_{0}\cup L_{1}$ and $ L_{1}$ is a hom-Lie subalgebra of $L.$

\begin{definition} \label{d1.1.18}\,   Let $(L, [,]_{L},\alpha)$ be a multiplicative hom-Lie algebra over $\mathbb{F}.$  A mapping
$[p]:L_{1}\rightarrow L_{1}, a\mapsto a^{[p]}$ is called a $p$-mapping, if

$\mathrm{(1)}$  $[\alpha(y), x^{[p]}]=(\mathrm{ad}x)^{p}(y),\ \forall x\in L_{1}, y\in L,$

$\mathrm{(2)}$  $(k x)^{[p]}=k^{p}x^{[p]},\ \forall x\in L_{1}, k\in \mathbb{F},$

$\mathrm{(3)}$
$(x+y)^{[p]}=x^{[p]}+y^{[p]}+\sum\limits_{i=1}^{p-1}s_{i}(x,y),$\\
where $ (\mathrm{ad}(x\otimes X+y\otimes 1))^{p-1}(x\otimes
1)=\sum\limits_{i=1}^{p-1}is_{i}(x,y)\otimes X^{i-1} $ in
$L\otimes_{\mathbb{F}}\mathbb{F}[X], \forall x,y\in L_{1}, \alpha(x\otimes X)=\alpha(x)\otimes X.$ The pair
$(L, [,]_{L},\alpha, [p])$ is referred to as a restricted hom-Lie algebra.
\end{definition}
From the above definition, we may see that
(i) $\alpha(x^{[p]})=(\alpha(x))^{[p]}$ for all $x\in L_{1},$ i.e., $\alpha\circ [p]=[p]\circ \alpha;$
(ii) By (1) of the definition, one gets $\mathrm{ad}x^{[p]}=(\mathrm{ad}x)^{p}$ for all $x\in L_{1}.$

Let $(L,\alpha)$ be a hom-Lie algebra over $\mathbb{F}$ and $f:L\rightarrow L$ be a mapping. $f$ is called a $p$-semilinear mapping, if
 $f(k x+y)=k^{p}f(x)+f(y), \ \forall x,y \in L, \ \forall k\in \mathbb{F}.$ Let $S$ be a subset of a
hom-Lie algebra $(L,\alpha).$ We put $C_{L}(S):=\{x\in L|\ [\alpha(y),x]=0, \forall y\in S\}.$
$C_{L}(S)$ is called the centralizer of $S$ in $L$.
Put $C(L):=\{x\in L|\ [\alpha(y),x]=0, \forall y\in L\}.$ $C(L)$ is called the center of $L.$

\begin{definition} Let $(L,[.,.]_{L},\alpha)$ be a restricted hom-Lie algebra over $\mathbb{F}.$ A hom-Lie subalgebra $H$ of $L$ is called a $p$-subalgebra,
if $x^{[p]}\in H_{1}$ for all $x\in H_{1},$ where $H_{1}=\{x\in H|\alpha(x)=x\}.$
\end{definition}
\begin{proposition} \label{p1.3.1} Let $L$ be a hom-Lie subalgebra of a restricted hom-Lie
algebra $(G,[,]_{G},\alpha, [p])$ and $[p]_{1}:L_{1}\rightarrow L_{1}$ a mapping. Then the following
statements are equivalent:

$\mathrm{(1)}$ $[p]_{1}$ is a $p$-mapping on $L_{1}.$

$\mathrm{(2)}$ There exists a $p$-semilinear mapping $f:L_{1}\rightarrow C_{G}(L)$
such that $[p]_{1}=[p]+f.$
\end{proposition}
\begin{proof}
(1)$\Rightarrow $(2). Consider $f:L_{1}\rightarrow G,$ $f(x)=x^{[p]_{1}}-x^{[p]}.$
Since $\mathrm{ad}f(x)(y)=[\alpha(y), f(x)]=0, \forall x\in L_{1}, y\in L,$ $f$ actually maps $L_{1}$ into $C_{G}(L).$
For $x,y\in L_{1}, k\in \mathbb{F},$ we obtain
\begin{eqnarray*}&&f(k x+y)\\
&&=k^{p}x^{[p]_{1}}+y^{[p]_{1}}+\sum_{i=1}^{p-1}s_{i}(k x,y)-k^{p}x^{[p]}-y^{[p]}-\sum_{i=1}^{p-1}s_{i}(k x,y)\\
&&=k^{p}f(x)+f(y),
\end{eqnarray*}
which proves that $f$ is $p$-semilinear.

(2)$\Rightarrow $(1). We next will check three conditions of the definition step and step.
For $x,y\in L_{1},$ we have
\begin{eqnarray*}&&
(x+y)^{[p]_{1}}
=(x+y)^{[p]}+f(x+y)\\
&&=x^{[p]}+f(x)+y^{[p]}+f(y)+\sum_{i=1}^{p-1}s_{i}(x,y)\\
&&=x^{[p]_{1}}+y^{[p]_{1}}+\sum_{i=1}^{p-1}s_{i}(x,y)
\end{eqnarray*}
and
\begin{eqnarray*}&&
(k x)^{[p]_{1}}=(k x)^{[p]}+f(k x)\\
&&=k^{p}x^{[p]}+k^{p}f(x)\\
&&=k^{p}(x^{[p]}+f(x))\\
&&=k^{p}x^{[p]_{1}}.
\end{eqnarray*}
 For $x\in L_{1}, z\in L,$ one gets
 \begin{eqnarray*}&&
\mathrm{ad}x^{[p]_{1}}(z)=\mathrm{ad}(x^{[p]}+f(x))(z)\\
&&=\mathrm{ad}x^{[p]}(z)+\mathrm{ad}f(x)(z)\\
&&=\mathrm{ad}x^{[p]}(z)\\
&&=(\mathrm{ad}x)^{p}(z).
\end{eqnarray*}
The proof is complete.
\end{proof}

\begin{corollary} \label{c1.3.3} The following statements hold.

$\mathrm{(1)}$ If $C(L)=0,$ then $L$ admits at most one $p$-mapping.

$\mathrm{(2)}$ If two $p$-mappings coincide on a basis, then they are equal.

$\mathrm{(3)}$ If $(L, [,]_{L},\alpha, [p])$ is restricted, then there exists a $p$-mapping $[p]^{'}$ of $L$ such that
$x^{[p]^{'}}=0, \ \forall x\in C(L_{1}).$
\end{corollary}

\begin{proof}
(1) We set $G=L.$ Then $C_{G}(L)=C(L),$ the only $p$-semilinear mapping
occurring in Proposition \ref{p1.3.1} is the zero mapping.

(2) If two $p$-mappings coincide on a basis, their difference vanishes since it is $p$-semilinear.

(3) $[p]|_{C(L_{1})}$ defines a $p$-mapping on $C(L_{1}).$ Since $C(L_{1})$ is abelian, it is $p$-semilinear.
Extend this to a $p$-semilinear mapping $f:L_{1}\rightarrow C(L_{1}).$ Then $[p]^{'}:=[p]-f$ is a $p$-mapping
of $L,$ vanishing on $C(L_{1}).$
\end{proof}

From the proof of Theorem 2 in \cite{dy}, we see the following definition:
\begin{definition} \label{d1.3.4001}
Let $(L, [.,.]_{L}, \alpha_{L})$ be a hom-Lie algebra, and let $j : L\rightarrow U_{HLie}(L)$ be the composition
of the maps $L
\hookrightarrow F_{HNAs}(L)\twoheadrightarrow U_{HLie}(L).$ The pair $(U_{HLie}(L),j)$ is called a universal enveloping algebra of $L$ if for every hom-associative
algebra $(A, \mu_{A}, \alpha_{A})$ and
every morphism $f : L\rightarrow HLie(A)$ of hom-Lie algebras, there exists a unique morphism $h: U_{HLie}(L)\rightarrow A$ of hom-associative
algebras such that $f = h \circ j$ (as morphisms of $\mathbb{F}$-modules).
\end{definition}
In the special case of $G=U_{HLie}(L)^{-}\supset L,$ where $U_{HLie}(L)$ is the universal enveloping algebra of hom-Lie algebra $L$ (see \cite{dy}) and
$U_{HLie}(L)^{-}$ denotes a hom-Lie algebra given by hom-associative
algebra $U_{HLie}(L)$ via the commutator bracket. We have the following theorem:

\begin{theorem} \label{t1.3.4} Let $(e_{j})_{j\in J}$ be a basis of $L_{1}$
such that there are $y_{j}\in L_{1}$ with $(\mathrm{ad}e_{j})^{p}=\mathrm{ad}y_{j}.$ Then there exists exactly
one $p$-mapping $[p]:L_{1}\rightarrow L$ such that $e_{j}^{[p]}=y_{j},\forall j\in J.$
\end{theorem}
\begin{proof}
For $z\in L_{1},$ we have
$0=((\mathrm{ad}e_{j})^{p}-\mathrm{ad}y_{j})(z)=[\alpha(z), e^{p}_{j}-y_{j}].$ Then
$e^{p}_{j}-y_{j}\in C_{U_{HLie}(L_{1})}{(L_{1})}, \forall j\in J.$ We define a
$p$-semilinear mapping $f:L_{1}\rightarrow C_{U_{HLie}(L_{1})}{(L_{1})}$ by means of
$$f(\sum \alpha_{j}e_{j}):=\sum\alpha_{j}^{p} (y_{j}-e^{p}_{j}).$$
Consider $V:=\{x\in L_{1}| x^{p}+f(x)\in L_{1}\}.$ The equation
$$(k x+y)^{p}+f(k x+y)=k^{p}x^{p}+y^{p}+\sum^{p-1}_{i=1}s_{i}(k x,y)+k^{p}f(x)+f(y)$$
ensures that $V$ is a subspace of $L_{1}.$ Since it contains the basis
$(e_{j})_{j\in J},$ we conclude that $x^{p}+f(x)\in L_{1},\ \forall x\in L_{1}.$
By virtue of Proposition \ref{p1.3.1},  $[p]:L_{1}\rightarrow L, x^{[p]}:=x^{p}+f(x)$ is a
$p$-mapping on $L_{1}.$ In addition, we obtain
$e^{[p]}_{j}=e^{p}_{j}+f(e_{j})=y_{j},$ as asserted. The
uniqueness of $[p]$ follows from Corollary \ref{c1.3.3}.
\end{proof}

\begin{definition} \label{d1.3.400} A multiplicative hom-Lie algebra $(L, [.,.]_{L}, \alpha_{L})$ is called restrictable, if $(\mathrm{ad}x)^{p}\in \mathrm{ad}{L_{1}}$ for all $x\in L_{1},$ where $\mathrm{ad}{L_{1}}=\{\mathrm{ad}x | x\in L_{1}\}.$
\end{definition}
\begin{theorem} \label{t1.3.400} $L$ is a restrictable hom-Lie algebra if and only if there is a $p$-mapping $[p]:L_{1}\rightarrow L_{1}$
which makes $L$ a  restricted hom-Lie algebra.
\end{theorem}
\begin{proof} $(\Leftarrow)$ By the definition of $p$-mapping $[p],$ for $x\in L_{1},$ there exists $x^{[p]}\in L_{1}$
such that $(\mathrm{ad}x)^{p}=\mathrm{ad}x^{[p]}\in \mathrm{ad}L_{1}.$ Hence $L$ is restrictable.

$(\Rightarrow)$ Let $L$ be restrictable. Then for $x\in L_{1},$  we have $(\mathrm{ad}{x})^{p}\in \mathrm{ad}L_{1},$
that is, there exists $y\in L_{1}$  such that $(\mathrm{ad}{x})^{p}=\mathrm{ad}{y}.$
Let $(e_{j})_{j\in J}$ be a basis of $L_{1}.$ Then there exists $y_{j}\in L_{1}$ such that $(\mathrm{ad}e_{j})^{p}=\mathrm{ad}y_{j}(j\in J).$
By Theorem  \ref{t1.3.4}, there exists exactly one $p$-mapping $[p]:L_{1}\rightarrow L_{1}$ such that
$e_{j}^{[p]}=y_{j}, \forall j\in J,$ which makes $L$ a restricted  hom-Lie algebra.
\end{proof}

\section{Properties of $p$-mappings and restrictable hom-Lie algebras}
In the section, we will discuss some properties of $p$-mappings and restrictable hom-Lie algebras.
\begin{definition}{\rm\cite{syh}}\, Let $(L,[\cdot,\cdot]_L,\alpha)$ and $(\Gamma,[\cdot,\cdot]_\Gamma,\beta)$ be two hom-Lie algebras. A linear map $\phi:L \rightarrow \Gamma$ is said to be a morphism of hom-Lie algebras if \\
\begin{equation} \phi[u,v]_L=[\phi(u),\phi(v)]_\Gamma,\qquad\forall  u,v\in L,\label{1}\end{equation}
\begin{equation}\phi\circ\alpha=\beta\circ\phi.\qquad\qquad\qquad\qquad\label{2}\end{equation}

 Denote by $\mathfrak{G}_\phi=\{(x,\phi(x))|x\in L\}\subseteq L\oplus\Gamma$ the graph of a linear map $\phi:L \rightarrow \Gamma$.
\end{definition}

\begin{definition}A morphism of hom-Lie algebras $\phi:(L,[\cdot,\cdot]_L,\alpha,[p]_{1})\rightarrow(\Gamma,[\cdot,\cdot]_\Gamma,\beta,[p]_{2})$ is said to be restricted if
$\phi(x^{[p]_{1}})=(\phi(x))^{[p]_{2}}$ for all $x\in L.$
\end{definition}

\begin{proposition}\label{proposition2.1}
Given two restricted hom-Lie algebras $(L,[\cdot,\cdot]_L,\alpha,[p]_{1})$ and $(\Gamma,[\cdot,\cdot]_\Gamma,\beta,$\\$[p]_{2})$, there is a restricted hom-Lie algebra
$(L\oplus\Gamma,[\cdot,\cdot]_{L\oplus\Gamma},\alpha+\beta,[p])$, where the bilinear map $[\cdot,\cdot]_{L\oplus\Gamma}:{\wedge}^2(L\oplus\Gamma)\rightarrow L\oplus\Gamma$ is given by
$${[u_1+v_1,u_2+v_2]}_{L\oplus\Gamma}={[u_1,u_2]}_L+{[v_1,v_2]}_\Gamma, \qquad \forall  u_1,u_2\in L,  v_1,v_2 \in \Gamma, $$
and the linear map $(\alpha+\beta):L\oplus\Gamma \rightarrow L\oplus\Gamma$ is given by
$$(\alpha+\beta)(u+v)=\alpha(u)+\beta(v), \qquad \forall  u\in L, v\in \Gamma,$$
 the $p$-mapping $[p]: L\oplus\Gamma \rightarrow L\oplus\Gamma$ is given by
$$(u+v)^{[p]}=u^{[p]_{1}}+v^{[p]_{2}}, \qquad \forall  u\in L, v\in \Gamma.$$
\end{proposition}
\begin{proof} Recall that $L_{1}=\{x\in L|\alpha(x)=x\}$ and $\Gamma_{1}=\{x\in \Gamma|\beta(x)=x\}.$
 For any  $u_1, u_2\in L, v_1, v_2\in \Gamma,$ we have
\begin{eqnarray*}
[u_2+v_2,u_1+v_1]_{L\oplus\Gamma}&=&{[u_2,u_1]}_L+{[v_2,v_1]}_\Gamma\\
&=&-{[u_1,u_2]}_L-{[v_1,v_2]}_\Gamma\\
&=&-[u_1+v_1,u_2+v_2]_{L\oplus\Gamma}.
\end{eqnarray*}
The bracket is obviously skew-symmetric. By a direct computation we
have
\begin{eqnarray*}
&&[(\alpha+\beta)(u_1+v_1),[u_2+v_2,u_3+v_3]_{L\oplus\Gamma}]_{L\oplus\Gamma}\\
&&+c.p.((u_1+v_1),(u_2+v_2),(u_3+v_3))\\
&=&[\alpha(u_1)+\beta(v_1),[u_2,u_3]_{L}+[v_2,v_3]_{\Gamma}]_{L\oplus\Gamma}+c.p.\\
&=&[\alpha(u_1),[u_2,u_3]_L]_L+c.p.(u_1,u_2,u_3)+[\beta(v_1),[v_2,v_3]_\Gamma]_\Gamma\\
&&+c.p.(v_1,v_2,v_3)\\
&=&0,
\end{eqnarray*}
where $c.p.(a,b,c)$ means the cyclic permutations of $a, b, c$.
For any  $u_1\in L_{1}, v_1\in \Gamma_{1}, u_2\in L, v_2\in \Gamma,$ we obtain
\begin{eqnarray*}
&&\mathrm{ad}(u_1+v_1)^{[p]}(u_2+v_2)=[(\alpha+\beta)(u_2+v_2),(u_1+v_1)^{[p]}]_{L\oplus\Gamma}\\
&&=[\alpha(u_2)+\beta(v_2),u_1^{[p]_{1}}+v_1^{[p]_{2}}]_{L\oplus\Gamma}\\
&&=[\alpha(u_2),u_1^{[p]_{1}}]_L+[\beta(v_2),v_1^{[p]_{2}}]_\Gamma\\
&&=\mathrm{ad}u_1^{[p]_{1}}(u_2)+\mathrm{ad}v_1^{[p]_{2}}(v_2)\\
&&=(\mathrm{adu_1})^{p}(u_2)+(\mathrm{adv_1})^{p}(v_2)
\end{eqnarray*}
and
\begin{eqnarray*}
&&(\mathrm{ad}(u_1+v_1))^{p}(u_2+v_2)\\
&&=[[[\alpha^{p}(u_2)+\beta^{p}(v_2),\overbrace{u_1+v_1],u_1+v_1],\cdots,u_1+v_1}^{p}]_{L\oplus\Gamma}\\
&&=[[[\alpha^{p}(u_2),\overbrace{u_1],u_1],\cdots,u_1}^{p}]_L+
[[[\beta^{p}(v_2),\overbrace{v_1],v_1],\cdots,v_1}^{p}]_\Gamma\\
&&=(\mathrm{adu_1})^{p}(u_2)+(\mathrm{adv_1})^{p}(v_2).
\end{eqnarray*}
Hence $\mathrm{ad}(u_1+v_1)^{[p]}(u_2+v_2)=(\mathrm{ad}(u_1+v_1))^{p}(u_2+v_2),$ thus $\mathrm{ad}(u_1+v_1)^{[p]}=(\mathrm{ad}(u_1+v_1))^{p}.$
Moreover, for any  $u_1, u_2\in L_{1}, v_1, v_2\in \Gamma_{1},$ one gets
\begin{eqnarray*}
&&((u_1+v_1)+(u_2+v_2))^{[p]}=((u_1+u_2)+(v_1+v_2))^{[p]}=(u_1+u_2)^{[p]_{1}}+(v_1+v_2)^{[p]_{2}}\\
&&=u_1^{[p]}+u_2^{[p]}+\sum\limits_{i=1}^{p-1}s_{i}(u_1,u_2)+v_1^{[p]}+v_2^{[p]}+\sum\limits_{i=1}^{p-1}s_{i}(v_1,v_2)\\
&&=(u_1^{[p]}+v_1^{[p]})+(u_2^{[p]}+v_2^{[p]})+(\sum\limits_{i=1}^{p-1}s_{i}(u_1,u_2)+\sum\limits_{i=1}^{p-1}s_{i}(v_1,v_2))\\
&&=(u_1+v_1)^{[p]}+(u_2+v_2)^{[p]}+\sum\limits_{i=1}^{p-1}(s_{i}(u_1,u_2))+s_{i}(v_1,v_2))\\
&&=(u_1+v_1)^{[p]}+(u_2+v_2)^{[p]}+\sum\limits_{i=1}^{p-1}s_{i}((u_1,v_1)+(u_2,v_2))
\end{eqnarray*}
and
\begin{eqnarray*}
&&(k(u_1+v_1))^{[p]}=(ku_1+kv_1 )^{[p]}=(ku_1)^{[p]_{1}}+(kv_1)^{[p]_{2}}\\
&&=k^{p}u_1^{[p]_{1}}+k^{p}v_1^{[p]_{2}}=k^{p}(u_1^{[p]_{1}}+v_1^{[p]_{2}})\\
&&=k^{p}(u_1+v_1)^{[p]}.
\end{eqnarray*}
Therefore, $(L\oplus\Gamma, [\cdot,\cdot]_{L\oplus\Gamma},\alpha+\beta,[p])$ is a restricted hom-Lie algebra.
\end{proof}

\begin{proposition}\label{proposition2.1}
A linear map $\phi:(L,[\cdot,\cdot]_L,\alpha,[p]_{1})\rightarrow(\Gamma,[\cdot,\cdot]_\Gamma,\beta,[p]_{2})$ is a restricted morphism of restricted hom-Lie algebras if and only if the graph $\mathfrak{G}_\phi\subseteq L\oplus\Gamma$ is a restricted hom-Lie subalgebra of  $(L\oplus\Gamma,[\cdot,\cdot]_{L\oplus\Gamma},\alpha+\beta,[p])$.
\end{proposition}
\begin{proof}
Let
$\phi:(L,[\cdot,\cdot]_L,\alpha)\rightarrow(\Gamma,[\cdot,\cdot]_\Gamma,\beta)$
be a restricted morphism of restricted hom-Lie algebras. By (\ref{1}), we have
$$[u+\phi(u),v+\phi(v)]_{L\oplus\Gamma}=[u,v]_{L}+[\phi(u),\phi(v)]_{\Gamma}=[u,v]_{L}+\phi[u,v]_L.$$
Then the graph $\mathfrak{G}_\phi$ is closed under the bracket operation $[\cdot,\cdot]_{L\oplus\Gamma}$. Furthermore, by (\ref{2}), we have
$$(\alpha+\beta)(u+\phi(u))=\alpha(u)+\beta\circ\phi(u)=\alpha(u)+\phi\circ\alpha(u),$$
which implies that $(\alpha+\beta)(\mathfrak{G}_\phi)\subseteq \mathfrak{G}_\phi$. Thus, $\mathfrak{G}_\phi$ is a hom-Lie subalgebra of $(L\oplus\Gamma,[\cdot,\cdot]_{L\oplus\Gamma},\alpha+\beta)$.
Moreover, for $u+\phi(u)\in \mathfrak{G}_\phi,$ one gets $$(u+\phi(u))^{[p]}=u^{[p]_{1}}+(\phi(u))^{[p]_{2}}=u^{[p]_{1}}+\phi(u^{[p]_{1}})\in \mathfrak{G}_\phi.$$
Thereby, the graph $\mathfrak{G}_\phi\subseteq L\oplus\Gamma$ is a restricted hom-Lie subalgebra of  $(L\oplus\Gamma,[\cdot,\cdot]_{L\oplus\Gamma},\alpha+\beta,[p])$.

Conversely, if the graph $\mathfrak{G}_\phi\subseteq L\oplus\Gamma$ is a restricted hom-Lie subalgebra of
$(L\oplus\Gamma,[\cdot,\cdot]_{L\oplus\Gamma},\alpha+\beta,[p])$, then we have
$$[u+\phi(u),v+\phi(v)]_{L\oplus\Gamma}=[u,v]_{L}+[\phi(u),\phi(v)]_{\Gamma}\in\mathfrak{G}_\phi,$$
which implies that $$[\phi(u),\phi(v)]_{\Gamma}=\phi[u,v]_{L}.$$
Furthermore, $(\alpha+\beta)(\mathfrak{G}_{\phi})\subset \mathfrak{G}_{\phi}$ yields that
$$(\alpha+\beta)(u+\phi(u))=\alpha(u)+\beta\circ\phi(u)\in\mathfrak{G}_\phi,$$
which is equivalent to the condition $\beta\circ\phi(u)=\phi\circ\alpha(u)$, i.e. $\beta\circ\phi=\phi\circ\alpha.$ Therefore, $\phi$ is a morphism of restricted hom-Lie algebras.
Since $\mathfrak{G}_\phi$ is a restricted hom-Lie subalgebra of  $(L\oplus\Gamma,[\cdot,\cdot]_{L\oplus\Gamma},\alpha+\beta,[p]),$  we have
$$(u+\phi(u))^{[p]}=u^{[p]_{1}}+(\phi(u))^{[p]_{2}}\in \mathfrak{G}_\phi.$$ Thus, $(\phi(u))^{[p]_{2}}=\phi(u^{[p]_{1}})$ for $u\in L,$ i.e., $\phi$ is a restricted morphism.
\end{proof}

One advantage in considering restrictable hom-Lie algebras instead of restricted ones rests on the following theorem.
\begin{theorem}  \label{t1.3.6} Let $f:(L, [,]_{L},\alpha, [p]_{1})\rightarrow (L^{'}, [,]_{L^{'}},\beta, [p]_{2})$ be a surjective restricted morphism of hom-Lie algebras. If $L$ is restrictable, so is $L^{'}.$
\end{theorem}
\begin{proof} It follows from $f$ is a surjective mapping that $L^{'}=f(L).$  Then for $x\in L_{1},$ we have $\beta(f(x))=f(\alpha(x))=f(x)$ and $f(x)\in L^{'}_{1},$ where $L_{1}=\{x\in L|\alpha(x)=x\}$ and $L^{'}_{1}=\{x\in L^{'}|\beta(x)=x\}.$ For $y\in L,$ one gets \begin{eqnarray*}&&
(\mathrm{ad}f(x))^{p}(f(y))=(\mathrm{ad}f(x))^{p-1}[\beta(f(y)),f(x)]\\
&&=(\mathrm{ad}f(x))^{p-2}[[\beta^{2}(f(y)),\beta(f(x))],f(x)]\\
&&=[[[\beta^{p}f(y),\underbrace{f(x)],f(x)],\cdots,f(x)]}_{p}\\
&&=\beta^{p}[[[f(y),\underbrace{f(x)],f(x)],\cdots,f(x)]}_{p}\\
&&=\beta^{p}\circ f[[[y,\underbrace{x],x],\cdots,x]}_{p} =f[[[\alpha^{p}(y),\underbrace{x],x],\cdots,x]}_{p}\\
&&=f((\mathrm{ad}x)^{p}(y))=f((\mathrm{ad}x^{[p]_{1}})(y))=f[\alpha(y), x^{[p]_{1}}]\\
&&=f[\alpha(y), \alpha(x^{[p]_{1}})]=f\circ \alpha[y, x^{[p]_{1}}]=\beta \circ f[y,x^{[p]_{1}}]\\
&&=\beta[f(y),f(x^{[p]_{1}})]=[\beta(f(y)),\beta(f(x^{[p]_{1}}))]\\
&&=[\beta(f(y)),f(x^{[p]_{1}})]=\mathrm{ad}f(x^{[p]_{1}})(f(y))\\
&&=\mathrm{ad}(f(x))^{[p]_{2}}(f(y)).
\end{eqnarray*}
We have $(\mathrm{ad}f(x))^{p}=\mathrm{ad}(f(x))^{[p]_{2}}\in \mathrm{ad}{L_{1}}^{'}.$ Hence $L^{'}$ is restrictable.
\end{proof}

\begin{theorem} \label{t1.3.7} Let $A$ and $B$ be hom-Lie ideals of hom-Lie algebra $(L,[.,.]_{L},\alpha)$ such that $L=A\oplus B.$  Then $L$ is restrictable if and only if $A,B$ are restrictable.
\end{theorem}
\begin{proof}
$(\Leftarrow)$ If $A,B$ are restrictable, for $x\in L_{1}$ with $\alpha(x)=x,$ we may suppose that $x=x_{1}+x_{2},$ where $x_{1}\in A, x_{2}\in B.$ Then $\alpha(x_{1}+x_{2})=\alpha(x_{1})+\alpha(x_{2})=x_{1}+x_{2}.$ Since $A$ and $B$ are hom-Lie ideals, one gets $\alpha(x_{1})\in A, \alpha(x_{2})\in B.$ we obtain $\alpha(x_{1})=x_{1}$ and $\alpha(x_{2})=x_{2}.$  As $A,B$ are restrictable, then there exists $y_{1}\in A_{1}, y_{2}\in B_{1}$ with $\alpha(y_{1})=y_{1}$ and $\alpha(y_{2})=y_{2},$
such that
$(\mathrm{ad}x_{1})^{p}=\mathrm{ad}y_{1}$ and $(\mathrm{ad}x_{2})^{p}=\mathrm{ad}y_{2}.$
 Thus, \begin{eqnarray*}
&&(\mathrm{ad}(x_{1}+x_{2}))^{p}=(\mathrm{ad}x_{1}+\mathrm{ad}x_{2})^{p}\\
&&=(\mathrm{ad}x_{1})^{p}+(\mathrm{ad}x_{2})^{p}=\mathrm{ad}y_{1}+\mathrm{ad}y_{2}\\
&&=\mathrm{ad}(y_{1}+y_{2}).
\end{eqnarray*}
Therefore, $L$ is restrictable.

 $(\Rightarrow)$
 If $L$ is restrictable , so are $A\cong L/B,$ $B\cong L/A$ by Theorem \ref{t1.3.6}.
 \end{proof}

 \begin{corollary} \label{c1.3.9} Let $A,B$ be restrictable hom-Lie ideals of a restricted hom-Lie algebra $(L,[.,.]_{L},\alpha,[p])$ such that $L=A+B$
 and $[A,B]=0.$ Then $L$ is restrictable.
\end{corollary}
\begin{proof}
 Define a mapping $f:A\oplus B\rightarrow L, (x,y)\mapsto x+y.$
Clearly, $f$ is a surjection.
For $(x_{1},y_{1}), (x_{2}, y_{2})\in A\oplus B,$
by $[A,B]=0,$ one gets $[x_{1},y_{2}]=[y_{1},x_{2}]=0.$ We have
\begin{eqnarray*}&&
f[(x_{1},y_{1}), (x_{2},y_{2})]=f([x_{1},x_{2}], [y_{1}, y_{2}])\\
&&=[x_{1},x_{2}]+[y_{1},y_{2}]=[x_{1},x_{2}]+[x_{1},y_{2}]+[y_{1},x_{2}]+[y_{1},y_{2}]\\
&&=[x_{1}+y_{1},x_{2}+y_{2}]=[f(x_{1},y_{1}),f(x_{2},y_{2})].
\end{eqnarray*}
Moreover, one gets
\begin{eqnarray*}&&
\alpha\circ f(x,y)=\alpha(x+y)\\
&&=\alpha(x)+ \alpha(y)=f((\alpha(x), \alpha(y)))\\
&&=f\circ \alpha(x,y).
\end{eqnarray*}
Therefore, $\alpha\circ f=f\circ \alpha.$
For $x\in A, y\in B,$ $\alpha(x,y)=(x,y),$ we have
\begin{eqnarray*}&&
f((x,y)^{[p]})=f((x^{[p]_{1}},y^{[p]_{2}}))\\
&&=x^{[p]_{1}}+y^{[p]_{2}}=(x+y)^{[p]}\\
&&=(f(x,y))^{[p]}.
\end{eqnarray*}
Thus, $f$ is a restricted morphism. By Theorem \ref{t1.3.7}, we have $A\oplus B$ is restrictable.  By Theorem \ref{t1.3.6}, one gets $L$ is restrictable.
\end{proof}

\begin{definition} Let $(L,[.,.]_{L},\alpha)$ be a hom-Lie algebra and $\psi$ be a symmetric bilinear form on $L.$
$\psi$ is called associative, if $\psi(x,[z,y])=\psi([\alpha(z),x],y).$
\end{definition}

\begin{definition} Let $(L,[.,.]_{L},\alpha)$ be a hom-Lie algebra and $\psi$ a symmetric bilinear form on $L.$
Set $L^{\bot}=\{x\in L| \psi(x,y)=0, \forall \ y\in L\}.$ $L$ is called nondegenerate, if $L^{\bot}=0.$
\end{definition}

\begin{theorem} \label{l1.3.16} Let $L$ be a subalgebra of the restricted hom-Lie algebra $(G,[.,.]_{G},\alpha,[p])$ with $C(L)=\{0\}.$  Assume
 $\lambda:G\times G\rightarrow \mathbb{F}$ to be an associative symmetric bilinear form, which is nondegenerate on $L\times L.$
 Then $L$ is restrictable.
\end{theorem}

\begin{proof} Since $\lambda$ is nondegenerate on $L\times L,$ every linear form $f$ on $L$ is determined
by a suitably chosen element $y\in L:f(z)=\lambda(y,z), \forall z\in L.$ Let $x\in L_{1}.$
Then there exists $y\in L$ such that
$$\lambda(x^{[p]},z)=\lambda(y,z),\forall z\in L.$$
This implies that $0=\lambda (x^{[p]}-y,[L,L])=\lambda ([\alpha(L), x^{[p]}-y],L)$ and $[\alpha(L), x^{[p]}-y]=0.$ Therefore, $x^{[p]}-y\in C(L)=\{0\}$ and $y=x^{[p]}\in L_{1}.$
Moreover, we obtain
$$(\mathrm{ad}x|_{L})^{p}=\mathrm{ad}x^{[p]}|_{L}=\mathrm{ad}y|_{L},$$
which proves that $L$ is  restrictable.
\end{proof}

\begin{proposition} \label{p1.3.14} Let $(L,[.,.]_{L},\alpha)$ be a restrictable hom-Lie algebra and $H$ a subalgebra of $L.$ Then $H$ is a
$p$-subalgebra for some mapping $[p]$ on $L$ if and only if $(\mathrm{ad}H_{1}|_{L})^{p}\subseteq \mathrm{ad}H_{1}|_{L}.$
\end{proposition}
\begin{proof}
$(\Rightarrow)$ If $H$ is a
$p$-subalgebra, then for $x\in H_{1},$ $x^{[p]}\in H_{1},$ and
$(\mathrm{ad}x)^{p}=\mathrm{ad}x^{[p]}\subseteq \mathrm{ad}H_{1}|_{L}.$
Hence, $(\mathrm{ad}H_{1}|_{L})^{p}\subseteq \mathrm{ad}H_{1}|_{L}.$

$(\Leftarrow)$ If $(\mathrm{ad}H_{1}|_{L})^{p}\subseteq \mathrm{ad}H_{1}|_{L},$ then $H$ is restrictable. By Theorem  \ref{t1.3.400},
$H$ is restricted. Thereby, $H$ is a $p$-subalgebra of $L.$
\end{proof}

\section{Cohomology of restricted hom-Lie algebras}

In this section, we will discuss the cohomology of restricted hom-Lie algebras in the abelian case, which is similar to the reference \cite{etj}.
The following definition is analogous to that of the restricted universal enveloping algebra in the reference \cite{sf}.
\begin{definition} \label{d1.3.4002}
Let $(L, [.,.]_{L}, \alpha, [p])$ be a restricted hom-Lie algebra. The pair $(u(L),\alpha^{'}, i)$ consisting of a hom-associative
algebra $(u(L),\alpha^{'})$ with unity and a restricted hom-morphism $i: L\rightarrow u(L)^{-}$ is called a restricted hom-universal enveloping algebra of $L$ if given any hom-associative
algebra $(A, \beta)$ with unity and
any restricted hom-morphism $f : L\rightarrow A^{-},$ there exists a unique morphism $\bar{f}: u(L)\rightarrow A$ of hom-associative
algebras such that $\bar{f}\circ i=f.$
\end{definition}

\begin{definition} {\rm\cite{ma}}\, \label{d1.3.4003}
Let $A=(V, \mu, \alpha)$ be a hom-associative $\mathbb{F}$-algebra. An $A$-module is a triple $(M,f,\gamma)$ where $M$ is $\mathbb{F}$-vector space and
$f,\gamma$ are $\mathbb{F}$-linear maps, $f:M\longrightarrow M$ and $\gamma: V\otimes M\longrightarrow M,$ such that the following diagram commutes:
\begin{displaymath}
\begin{CD}
 V\otimes M @ > \gamma>>  M \\
@AA{\alpha \otimes \gamma}A  @AA {\gamma} A\\
V\otimes V\otimes M             @> \mu \otimes f>> V\otimes M
\end{CD}
\end{displaymath}
\end{definition}

We let $S^{\ast}(L)$ and $\Lambda^{\ast}(L)$ denote the symmetric and alternating algebras of restricted hom-Lie algebra $(L, [.,.]_{L}, \alpha, [p])$, respectively.
Bases for the homogeneous subspaces of degree $k$ for these spaces consist
of monomials $e^{\mu} = e_{1}^{\mu_{1}}\cdots e_{n}^{\mu_{n}}$ and $e_{\vec{i}}=e_{i_{1}}\wedge\cdots\wedge e_{i_{k}},$ respectively,
where

$\mu=(\mu_{1},\cdots,\mu_{n})\in \mathbb{Z}^{n}$ satisfies $\mu_{j}\geq 0, |\mu|=\sum_{j}\mu_{j}=k;$

$\vec{i}=(i_{1},\cdots, i_{k})\in \mathbb{Z}^{k}$ satisfies $1\leq i_{1}<\cdots<i_{k}\leq n.$\\
Let $\alpha: \lambda\mapsto \lambda^{p}$ denote the Frobenius automorphism of $\mathbb{F}.$ If $V$ is an abelian group
with an $\mathbb{F}$-vector space structure given by $\mathbb{F}\rightarrow \mathrm{End}(V ),$ then the composition
$$\mathbb{F}\xrightarrow{\alpha^{-1}}\mathbb{F}\rightarrow \mathrm{End}(V )$$
gives another vector space structure on $V$ which we will denote by $\overline{V}.$ Of course
$\overline{V}$ is isomorphic to $V$ as an $\mathbb{F}$-vector space (they have the same dimension). We
note that if $W$ is any other $\mathbb{F}$-vector space, then a $p$-semilinear map $V\rightarrow W$
is a linear map $\overline{V}\rightarrow W$ and vice versa.

In sequel, $(L, [.,.]_{L}, \alpha,[p])$ denotes a finite-dimensional  restricted
hom-Lie algebra over $\mathbb{F}$ such that $[g_{i}, g_{j}]=0$ for all $g_{i},g_{j}\in L$ and $(u(L), \alpha^{'}, i)$ denotes the
restricted hom-universal enveloping algebra of $L.$ Here we take $\alpha=\alpha^{'}$ and $\alpha$ satisfies $\alpha(u_{1}u_{2})=\alpha(u_{1})\alpha(u_{2})$ for $u_{1},u_{2}\in u(L).$ For $s,t\geq 0,$ we define
$$C_{s,t} = S^{t}\overline{L}_{1}\otimes \Lambda^{s}L\otimes u(L)$$
with the $u(L)$-module structure given by
$$u(h_{1}\cdots h_{t}\otimes g_{1}\wedge\cdots \wedge g_{s}\otimes x) =h_{1}\cdots h_{t}\otimes g_{1}\wedge \cdots \wedge g_{s}\otimes \alpha(u)x,$$
where $h_{i}, g_{j}\in L$ and $u,x\in u(L).$ If either $s<0$ or $t<0,$ we put $C_{s,t}=0$ and
define
$$C_{k} =
\bigoplus_{2t+s=k}C_{s,t}$$
for all $k\in \mathbb{Z}.$ Note that each $C_{k}$ is a free $u(L)$-module. If not both $t=0$ and
$s=0,$ we then define a map
$$d_{s,t}: C_{s,t}\rightarrow C_{t,s-1}\oplus C_{t-1,s+1}$$ by the formulas

\quad \quad \quad \quad \quad $d_{t,s}(h_{1}\cdots h_{t}\otimes g_{1}\wedge \cdots \wedge g_{s}\otimes x)$
\begin{eqnarray}
&&=\sum_{i=1}^{s}(-1)^{i-1}h_{1}\cdots h_{t}\otimes \alpha(g_{1})\wedge \cdots\widehat{\alpha(g_{i})}\cdots \wedge \alpha(g_{s})\otimes \alpha(g_{i})x\label{0.5} \\&&+\sum_{j=1}^{t}h_{1}\cdots \widehat{h_{j}}\cdots h_{t}\otimes h_{j}^{[p]}\wedge \alpha(g_{1})\wedge \cdots \wedge \alpha(g_{s})\otimes \alpha(x)\label{0.6}\\
&&-\sum_{j=1}^{t}h_{1}\cdots \widehat{h_{j}}\cdots h_{t}\otimes h_{j}\wedge \alpha(g_{1})\wedge \cdots \wedge \alpha(g_{s})\otimes h_{j}^{p-1}x\label{0.7}.
\end{eqnarray}
For $k\geq 1,$ we define the map $d_{k}: C_{k}\rightarrow
C_{k-1}$ by $d_{k}=\bigoplus_{2t+s=k}d_{s,t}.$
Then we obtain the following theorem.
\begin{theorem} \label{t1.3.160} The maps $d_{k}$ defined above satisfy $d_{k-1}d_{k}=0$ for $k\geq 1,$ so that $C=(C_{k},d_{k})$ is an augmented complex of free $u(g)$-modules.
\end{theorem}
\begin{proof} The terms in the sum (\ref{0.5}) are elements of $C_{t,s-1}$ whereas the terms in the
sums (\ref{0.6}) and (\ref{0.7}) lie in $C_{t-1,s+1}.$ Therefore, in order to compute $d_{k-1}d_{k}=0,$ we must
apply $d_{t,s-1}$ to (\ref{0.5}) and $d_{t-1,s+1}$ to (\ref{0.6}) and (\ref{0.7}). Applying $d_{t,s-1}$ to (5), we have

\begin{eqnarray*}&&d_{t,s}\big(\sum_{i=1}^{s}(-1)^{i-1}h_{1}\cdots h_{t}\otimes \alpha(g_{1})\wedge \cdots \widehat{\alpha(g_{i})} \cdots \wedge \alpha(g_{s})\otimes \alpha(g_{i})x \big)
\\&=&\sum_{i=1}^{s}(-1)^{i-1}\big(\sum_{\sigma<i}(-1)^{\sigma-1}h_{1}\cdots h_{t}\otimes \alpha^{2}(g_{1})\wedge \cdots \widehat{\alpha^{2}(g_{\sigma})}\cdots \widehat{\alpha^{2}(g_{i})}\cdots \wedge \alpha^{2}(g_{s})
\\&&\quad \otimes \alpha^{2}(g_{\sigma})(\alpha(g_{i})x)\\
&&+\sum_{\sigma>i}(-1)^{\sigma}h_{1}\cdots h_{t}\otimes \alpha^{2}(g_{1})\wedge \cdots \widehat{\alpha^{2}(g_{i})}\cdots \widehat{\alpha^{2}(g_{\sigma})}\cdots \wedge \alpha^{2}(g_{s})\otimes \alpha^{2}(g_{\sigma})(\alpha(g_{i})x)\\
&&+\sum_{j=1}^{t}h_{1}\cdots \widehat{h_{j}}\cdots h_{t}\otimes h_{j}^{[p]}\wedge \alpha^{2}(g_{1})\wedge \cdots \widehat{\alpha^{2}(g_{i})}\cdots \wedge \alpha^{2}(g_{s})\otimes \alpha(\alpha(g_{i})x)\\
&&-\sum_{j=1}^{t}h_{1}\cdots \widehat{h_{j}}\cdots h_{t}\otimes h_{j}\wedge \alpha^{2}(g_{1})\wedge \cdots \widehat{\alpha^{2}(g_{i})}\cdots \wedge \alpha^{2}(g_{s})\otimes h_{j}^{p-1}(\alpha(g_{i})x) \big)\\
&=&\sum_{i=1}^{s}(-1)^{i-1}\big(\sum_{\sigma<i}(-1)^{\sigma-1}h_{1}\cdots h_{t}\otimes \alpha^{2}(g_{1})\wedge \cdots \widehat{\alpha^{2}(g_{\sigma})}\cdots \widehat{\alpha^{2}(g_{i})}\cdots \wedge \alpha^{2}(g_{s})
\\&&\quad \otimes (\alpha(g_{\sigma})\alpha(g_{i}))\alpha(x)\\
&&+\sum_{\sigma>i}(-1)^{\sigma}h_{1}\cdots h_{t}\otimes \alpha^{2}(g_{1})\wedge \cdots \widehat{\alpha^{2}(g_{i})}\cdots \widehat{\alpha^{2}(g_{\sigma})}\cdots \wedge \alpha^{2}(g_{s})\otimes (\alpha(g_{\sigma})\alpha(g_{i}))\alpha(x)
\end{eqnarray*}
\begin{eqnarray}
&+&\sum_{j=1}^{t}h_{1}\cdots \widehat{h_{j}}\cdots h_{t}\otimes h_{j}^{[p]}\wedge \alpha^{2}(g_{1})\wedge \cdots \widehat{\alpha^{2}(g_{i})}\cdots \wedge \alpha^{2}(g_{s})\otimes \alpha(\alpha(g_{i})x)\label{0.8}\\
&-&\sum_{j=1}^{t}h_{1}\cdots \widehat{h_{j}}\cdots h_{t}\otimes h_{j}\wedge \alpha^{2}(g_{1})\wedge \cdots \widehat{\alpha^{2}(g_{i})}\cdots \wedge \alpha^{2}(g_{s})\otimes (h_{j}^{p-1}\alpha(g_{i}))\alpha(x) \big)\label{0.9}.
\end{eqnarray}
Since $\alpha(g_{i})\alpha(g_{j})=\alpha(g_{j})\alpha(g_{i})$ in $u(g),$ the terms in the first two sums
in the parentheses cancel in pairs when summed over all $i.$ This leaves the sum over
$i$ of (\ref{0.8}) and (\ref{0.9}). Now we apply $d_{t-1,s+1}$ to (\ref{0.6}).

\begin{eqnarray*}d_{t-1,s+1}\big(\sum_{j=1}^{t}h_{1}\cdots \widehat{h_{j}}\cdots h_{t}\otimes h_{j}^{[p]}\wedge \alpha(g_{1})\wedge  \cdots \wedge \alpha(g_{s})\otimes \alpha(x) \big)=\sum_{j=1}^{t}
\end{eqnarray*}
\begin{eqnarray}&&\big(\sum_{\sigma=1}^{s}(-1)^{\sigma}h_{1}\cdots \widehat{h_{j}}\cdots h_{t}\otimes h_{j}^{[p]}\wedge \alpha^{2}(g_{1})\wedge \cdots \widehat{\alpha^{2}(g_{\sigma})}\cdots \wedge \alpha^{2}(g_{s})
\otimes \alpha^{2}(g_{\sigma})\alpha(x)\label{1.0}\\
&&+h_{1}\cdots \widehat{h_{j}}\cdots h_{t}\otimes \alpha^{2}(g_{1})\wedge \cdots \wedge \alpha^{2}(g_{s})\otimes \alpha(h_{j}^{[p]})\alpha(x)\label{1.1}\\
&&+\sum_{\tau\neq j}h_{1}\cdots \widehat{h_{\tau}}\cdots \widehat{h_{j}}\cdots h_{t}\otimes h_{\tau}^{[p]}\wedge h_{j}^{[p]}\wedge \alpha^{2}(g_{1})\wedge \cdots \wedge \alpha^{2}(g_{s})\otimes \alpha^{2}(x)\label{1.2}\\
&&-\sum_{\tau\neq j}h_{1}\cdots \widehat{h_{\tau}}\cdots \widehat{h_{j}}\cdots h_{t}\otimes h_{\tau}\wedge h_{j}^{[p]}\wedge \alpha^{2}(g_{1})\wedge \cdots \wedge \alpha^{2}(g_{s})\otimes h_{\tau}^{p-1}\alpha(x)\big)\label{1.3}.
\end{eqnarray}
We note that the terms in (\ref{1.2}) cancel in pairs since interchanging the first two terms
in the alternating product multiplies the term by $-1$. Finally, we apply $d_{t-1,s+1}$ to
(\ref{0.7}) to get
\begin{eqnarray*}&&d_{t-1,s+1}\big(-\sum_{j=1}^{t}h_{1}\cdots \widehat{h_{j}}\cdots h_{t}\otimes h_{j}\wedge \alpha^{2}(g_{1})\wedge  \cdots \wedge \alpha(g_{s})\otimes h_{j}^{p-1}x \big)
\\&&=-\sum_{j=1}^{t}\big(\sum_{\sigma=1}^{s}(-1)^{\sigma}h_{1}\cdots \widehat{h_{j}}\cdots h_{t}\otimes h_{j}\wedge \alpha^{2}(g_{1})\wedge \cdots \wedge \alpha^{2}(g_{s})\otimes \alpha^{2}(g_{\sigma})(h_{j}^{p-1}x) \\
&&\quad+h_{1}\cdots \widehat{h_{j}}\cdots h_{t}\otimes \alpha^{2}(g_{1})\wedge \cdots \wedge \alpha^{2}(g_{s})\otimes \alpha^{2}(h_{j})(h_{j}^{p-1}x)\\
&&\quad+\sum_{\tau\neq j}h_{1}\cdots \widehat{h_{\tau}}\cdots \widehat{h_{j}}\cdots h_{t}\otimes h_{\tau}^{[p]}\wedge h_{j}\wedge \alpha^{2}(g_{1})\wedge \cdots \wedge \alpha^{2}(g_{s})\otimes \alpha(h_{j}^{p-1}x)\\
&&\quad-\sum_{\tau\neq j}h_{1}\cdots \widehat{h_{\tau}}\cdots \widehat{h_{j}}\cdots h_{t}\otimes h_{\tau}\wedge h_{j}\wedge \alpha^{2}(g_{1})\wedge \cdots \wedge \alpha^{2}(g_{s})\otimes h_{\tau}^{p-1}(h_{j}^{p-1}x) \big)
\end{eqnarray*}
\begin{eqnarray}&&
=-\sum_{j=1}^{t}\big(\sum_{\sigma=1}^{s}(-1)^{\sigma}h_{1}\cdots \widehat{h_{j}}\cdots h_{t}\otimes h_{j}\wedge \alpha^{2}(g_{1})\wedge \cdots \wedge \alpha^{2}(g_{s})\otimes \alpha(g_{\sigma}h_{j}^{p-1})\alpha(x) \label{1.4}\\
&&\quad+h_{1}\cdots \widehat{h_{j}}\cdots h_{t}\otimes \alpha^{2}(g_{1})\wedge \cdots \wedge \alpha^{2}(g_{s})\otimes h_{j}^{p}\alpha(x) \label{1.5}\\
&&\quad+\sum_{\tau\neq j}h_{1}\cdots \widehat{h_{\tau}}\cdots \widehat{h_{j}}\cdots h_{t}\otimes h_{\tau}^{[p]}\wedge h_{j}\wedge \alpha^{2}(g_{1})\wedge \cdots \wedge \alpha^{2}(g_{s})\otimes \alpha(h_{j}^{p-1}x) \label{1.6}\\
&& \, -\sum_{\tau\neq j}h_{1}\cdots \widehat{h_{\tau}}\cdots \widehat{h_{j}}\cdots h_{t}\otimes h_{\tau}\wedge h_{j}\wedge \alpha^{2}(g_{1})\wedge \cdots \wedge \alpha^{2}(g_{s})\otimes (h_{\tau}^{p-1}h_{j}^{p-1})\alpha(x) \big) \label{1.7}.
\end{eqnarray}
This time the terms in (\ref{1.7}) cancel in pairs. Moreover, the terms in (\ref{0.8}) and (\ref{1.0})
are identical (with $\sigma=i$) except for sign and hence they cancel. The terms in (\ref{0.9})
and (\ref{1.4}) cancel in pairs since $\alpha(h_{i}^{p-1})
\alpha(g_{j})=\alpha(g_{j})\alpha(h_{i}^{p-1}).$ The terms in (\ref{1.3}) and (\ref{1.6})
have the same sign but are equal apart from interchanging the first two terms in the
alternating part. Finally the terms in (\ref{1.1}) and (\ref{1.5}) match except for sign since
$h_{j}^{[p]}=h_{j}^{p}$ in $u(g)$ and hence the entire sum is zero as claimed.
This completes the proof.
\end{proof}

We next will consider the cohomology  of restricted hom-Lie algebras in the case of simpleness.
A basis for the space $C_{t,s}$ consists of the monomials
$$e^{\mu}\otimes e_{I}\otimes e^{r}=e_{1}^{\mu_{1}}\cdots e_{n}^{\mu_{n}}\otimes e_{i_{1}}\wedge \cdots \wedge e_{i_{s}}\otimes e_{1}^{r_{1}}\cdots e_{n}^{r_{n}},
$$
where $\mu=(\mu_{1},\cdots,\mu_{n}), I=(i_{1},\cdots, i_{s}), r=(r_{1},\cdots,r_{n})$ and
$$\mu_{j}\geq 0, |\mu|=\sum_{j}\mu_{j}=t, 1\leq i_{1}<\cdots<i_{s}\leq n, 0\leq r_{j}\leq p-1.$$
For each $i=1,\cdots, n$ and $e_{i}\in L_{1},$ we let $$c_{i}=1\otimes e_{i}^{[p]}\otimes 1-1\otimes e_{i}\otimes e_{i}^{p-1}$$
and we easily note that $c_{i}\in C_{0,1}$ is a cycle for all $i.$ Now we define
$$(\partial/\partial e_{i}\otimes c_{i}):C_{t,s}\longrightarrow C_{t-1,s+1}$$
by the formula
$$(\frac{\partial}{\partial e_{i}}\otimes c_{i})(e^{\mu}\otimes e_{I}\otimes e^{r})=\frac{\partial e^{\mu}}{\partial e_{i}}\otimes e_{i}^{[p]}\wedge \alpha(e_{I})\otimes \alpha(e^{r})-\frac{\partial e^{\mu}}{\partial e_{i}}\otimes e_{i}\wedge \alpha(e_{I})\otimes e_{i}^{p-1}\alpha(e^{r}).$$
If $\mu=(\mu_{1},\cdots,\mu_{n})$ satisfies $|\mu|=t$ and $I=(i_{1},\cdots,i_{s})$ is increasing, then by the definition we write
$$e^{\mu}\otimes c_{I}=\sum_{J\subset \{1,\cdots,s\}}(-1)^{|J|}e^{\mu}\otimes f_{i_{1}}\wedge \cdots \wedge f_{i_{s}}\otimes e_{i_{1}}^{q_{i_{1}}}\cdots e_{i_{s}}^{q_{i_{s}}}$$
and $$e^{\mu}\otimes \alpha(c_{I})=\sum_{J\subset \{1,\cdots,s\}}(-1)^{|J|}e^{\mu}\otimes \alpha(f_{i_{1}})\wedge \cdots \wedge \alpha(f_{i_{s}})\otimes \alpha(e_{i_{1}}^{q_{i_{1}}})\cdots \alpha(e_{i_{s}}^{q_{i_{s}}}),$$
where
\[
f_{i_{j}}=\left\{
\begin{array}{lll}
e_{i_{j}},& j\in J \\
e_{i_{j}}^{[p]},& j\notin J;
\end{array}
\right. \quad q_{i_{j}}=\left\{
\begin{array}{lll}
p-1,& j\in J  \\
0,& j\notin J.
\end{array}
\right.
\]
We then define $\mathfrak{C}_{t,s}$ to be the $\mathbb{F}$-subspace of $C_{t,s}$ spanned by the elements
$\{e^{\mu}\otimes \alpha(c_{I}): |\mu|=t$ and $I$ is increasing $\}$
and $$\mathfrak{C}_{k}=\bigoplus_{2t+s=k}\mathfrak{C}_{t,s}.$$
The boundary operator $\partial_{k}=\partial:\mathfrak{C}_{k}\longrightarrow \mathfrak{C}_{k-1}$ is defined by
$$\partial=\sum_{j=1}^{n}\frac{\partial}{\partial e_{j}}\otimes c_{j}.$$
Then we may show that $\partial^{2}=0.$
In fact,
\begin{eqnarray*}&&\partial^{2}(e^{\mu}\otimes c_{I})=\partial(\partial(e^{\mu}\otimes c_{I}))
\\&&=\partial(\sum_{j=1}^{n}\frac{\partial}{\partial e_{j}}\otimes c_{j}(\sum_{J\subset \{1,\cdots,s\}}(-1)^{|J|}e^{\mu}\otimes f_{i_{1}}\wedge \cdots \wedge f_{i_{s}}\otimes e_{i_{1}}^{q_{i_{1}}}\cdots e_{i_{s}}^{q_{i_{s}}}))
\\&&=\partial(\sum_{j=1}^{n}\sum_{J\subset \{1,\cdots,s\}}(-1)^{|J|}(\frac{\partial e^{\mu}}{\partial e_{j}}\otimes e_{j}^{[p]}\wedge \alpha(f_{i_{1}})\wedge \cdots \wedge \alpha(f_{i_{s}})\otimes \alpha(e_{i_{1}}^{q_{i_{1}}})\cdots \alpha(e_{i_{s}}^{q_{i_{s}}})\\&& \quad-\frac{\partial e^{\mu}}{\partial e_{j}}\otimes e_{j}\wedge \alpha(f_{i_{1}})\wedge \cdots \wedge \alpha(f_{i_{s}})\otimes e_{j}^{p-1}\alpha(e_{i_{1}}^{q_{i_{1}}})\cdots \alpha(e_{i_{s}}^{q_{i_{s}}})))
\\&&=\sum_{l=1}^{n}\sum_{j=1}^{n}\sum_{J\subset \{1,\cdots,s\}}(-1)^{|J|}\{\frac{\partial}{\partial e_{l}}\otimes c_{l}(\frac{\partial e^{\mu}}{\partial e_{j}}\otimes e_{j}^{[p]}\wedge \alpha(f_{i_{1}})\wedge \cdots \wedge \alpha(f_{i_{s}})
\otimes \alpha(e_{i_{1}}^{q_{i_{1}}})\cdots
\\&& \quad\quad\alpha(e_{i_{s}}^{q_{i_{s}}}))
-\frac{\partial}{\partial e_{l}}\otimes c_{l}(\frac{\partial e^{\mu}}{\partial e_{j}}\otimes e_{j}\wedge \alpha(f_{i_{1}})\wedge \cdots \wedge \alpha(f_{i_{s}})\otimes e_{j}^{p-1}\alpha(e_{i_{1}}^{q_{i_{1}}})\cdots \alpha(e_{i_{s}}^{q_{i_{s}}}))\}
\\&& \quad=\sum_{l=1}^{n}\sum_{j=1}^{n}\sum_{J\subset \{1,\cdots,s\}}(-1)^{|J|}\{\frac{\partial(\frac{\partial e^{\mu}}{\partial e_{j}})}{\partial e_{l}}\otimes e_{l}^{[p]}\wedge \alpha(e_{j}^{[p]})\wedge \alpha^{2}(f_{i_{1}})\wedge \cdots \wedge \alpha^{2}(f_{i_{s}})
\end{eqnarray*}
\begin{eqnarray}&&\qquad\qquad\qquad\qquad\qquad\quad\otimes \alpha^{2}(e_{i_{1}}^{q_{i_{1}}})\cdots \alpha^{2}(e_{i_{s}}^{q_{i_{s}}})\label{1.8}
\\&& \quad-\frac{\partial(\frac{\partial e^{\mu}}{\partial e_{j}})}{\partial e_{l}}\otimes e_{l}\wedge \alpha(e_{j}^{[p]})\wedge \alpha^{2}(f_{i_{1}})\wedge \cdots \wedge \alpha^{2}(f_{i_{s}})\otimes e_{l}^{p-1}\alpha^{2}(e_{i_{1}}^{q_{i_{1}}})\cdots \alpha^{2}(e_{i_{s}}^{q_{i_{s}}})\label{1.9}
\\&&\quad-\frac{\partial(\frac{\partial e^{\mu}}{\partial e_{j}})}{\partial e_{l}}\otimes e_{l}^{[p]}\wedge \alpha(e_{j})\wedge \alpha^{2}(f_{i_{1}})\wedge \cdots \wedge \alpha^{2}(f_{i_{s}})\otimes \alpha(e_{j}^{p-1})\alpha^{2}(e_{i_{1}}^{q_{i_{1}}})\cdots \alpha^{2}(e_{i_{s}}^{q_{i_{s}}})\label{2.0}
\end{eqnarray}
\begin{eqnarray*}&&\quad\quad \,+\frac{\partial(\frac{\partial e^{\mu}}{\partial e_{j}})}{\partial e_{l}}\otimes e_{l}\wedge \alpha(e_{j})\wedge \alpha^{2}(f_{i_{1}})\wedge \cdots \wedge \alpha^{2}(f_{i_{s}})\otimes e_{l}^{p-1}\alpha(e_{j}^{p-1})\alpha^{2}(e_{i_{1}}^{q_{i_{1}}})\cdots \alpha^{2}(e_{i_{s}}^{q_{i_{s}}})\}.
\end{eqnarray*}
\begin{eqnarray}\label{2.1}
\end{eqnarray}
This time the terms in (\ref{1.8}) cancel in pairs, and the terms in (\ref{2.1}) cancel in pairs since $e_{l}^{p-1}\alpha(e_{j}^{p-1})=\alpha(e_{l}^{p-1})e_{j}^{p-1}$. Moreover, the terms in (\ref{1.9}) and (\ref{2.0}) are identical except for sign and hence they cancel,
so that $\mathfrak{C}=\{\mathfrak{C}_{k},\partial_{k}\}_{k\geq 0}$ is a complex.

\begin{theorem} \label{t1.3.1600} If $\mathfrak{C}$ is the complex defined above, we define $H_{k}(\mathfrak{C}):=\mathrm{Ker}\partial_{k}/\mathrm{Im}\partial_{k}.$ Then
\[
H_{k}(\mathfrak{C})=\left\{
\begin{array}{lll}
U_{res.}(g), & k=0 \\
0,& 0<k<p.
\end{array}
\right.
\]
\end{theorem}
\begin{proof} Define a map $D: \mathfrak{C}_{k}\rightarrow \mathfrak{C}_{k+1}$ by the formula
$$D(e^{\mu}\otimes \alpha(c_{I}))=\sum\limits_{a=1}^{s}(-1)^{a-1}e^{\mu}e_{i_{a}}\otimes c_{i_{1}}\cdots \widehat{c_{i_{a}}}\cdots c_{i_{s}}$$
and compute for any monomial $e^{\mu}\otimes \alpha(c_{I}):$

$D\partial(e^{\mu}\otimes  \alpha(c_{I}))=D(\sum\limits_{j=1}^{n}(\frac{\partial}{\partial e_{j}}\otimes c_{j})(e^{\mu}\otimes  \alpha(c_{I})))$
\begin{eqnarray*}&&=\sum\limits_{j=1,{j\neq i_{1},\cdots,i_{s}}}^{n}D(\mu_{j}e_{1}^{\mu_{1}}\cdots e_{j}^{\mu_{j}-1}\cdots e_{n}^{\mu_{n}}\otimes c_{j}\alpha^{2}(c_{I}))
\\&&=\sum\limits_{j=1,{j\neq i_{1},\cdots,i_{s}}}^{n}D(\mu_{j}e_{1}^{\mu_{1}}\cdots e_{j}^{\mu_{j}-1}\cdots e_{n}^{\mu_{n}}\otimes \alpha(c_{j})\alpha^{2}(c_{I}))
\\&&=\big(\sum\limits_{{j=1,}{j\neq i_{1},\cdots,i_{s}}}^{n}\mu_{j}\big)e^{\mu}\otimes \alpha(c_{I})
\\&&+\sum\limits_{{j=1,}{j\neq i_{1},\cdots,i_{s}}}^{n}\sum_{a=1}^{s}(-1)^{a}\mu_{j}e_{1}^{\mu_{1}}\cdots e_{j}^{\mu_{j}-1}\cdots e_{i_{a}}^{\mu_{i_{a}}+1}\cdots e_{n}^{\mu_{n}}
\end{eqnarray*}
\begin{eqnarray} \otimes \alpha(c_{j})\alpha(c_{i_{1}})\cdots \widehat{\alpha(c_{i_{a}})}\cdots \alpha(c_{i_{s}})\label{2.2}
\end{eqnarray}
and
$\partial D(e^{\mu}\otimes \alpha(c_{I}))=\partial(\sum\limits_{a=1}^{s}(-1)^{a-1}e^{\mu}e_{i_{a}}\otimes c_{i_{1}}\cdots \widehat{c_{i_{a}}}\cdots c_{i_{s}})$
\begin{eqnarray*}&&=\sum\limits_{a=1}^{s}(-1)^{a-1}\partial(e_{1}^{\mu_{1}}\cdots e_{i_{a}}^{\mu_{i_{a}}+1}\cdots e_{n}^{\mu_{n}}\otimes c_{i_{1}}\cdots \widehat{c_{i_{a}}}\cdots c_{i_{s}})
\\&&=\big(\sum\limits_{a=1}^{s}\mu_{i_{a}}+1\big)e^{\mu}\otimes \alpha(c_{I})
\\&&-\sum_{a=1}^{s}(-1)^{a}\sum\limits_{{j=1,}{j\neq i_{1},\cdots,i_{s}}}^{n}\mu_{j}e_{1}^{\mu_{1}}\cdots e_{j}^{\mu_{j}-1}\cdots e_{i_{a}}^{\mu_{i_{a}}+1}\cdots e_{n}^{\mu_{n}}
\end{eqnarray*}
\begin{eqnarray}\otimes \alpha(c_{j})\alpha(c_{i_{1}})\cdots \widehat{\alpha(c_{i_{a}})}\cdots \alpha(c_{i_{s}})\label{2.3}.
\end{eqnarray}
Clearly the terms (\ref{2.2}) and (\ref{2.3}) are identical apart from sign so that we have
$$(D\partial+\partial D)(e^{\mu}\otimes \alpha(c_{I}))=\big(\sum\limits_{{j=1,}{j\neq i_{1},\cdots,i_{s}}}^{n}\mu_{j}+\sum\limits_{a=1}^{s}\mu_{i_{a}}+s\big)(e^{\mu}\otimes \alpha(c_{I}))=(t+s)(e^{\mu}\otimes \alpha(c_{I})).$$
Therefore we see that every cycle in $\mathfrak{C}_{k}(k=2t+s)$ is a boundary provided that $t+s\neq 0 (\mathrm{mod} p).$ In particular, if $0<k<p,$ then $0<t+s<p$
so that $H_{k}(\mathfrak{C})=0.$ Moreover, $\mathfrak{C}_{1}=\mathfrak{C}_{0,1}$ is spanned by the $c_{i}$ and $\partial c_{i}=0$ for all $i.$ Therefore $H_{0}(\mathfrak{C})=\mathfrak{C}_{0}=U_{res.}(g),$ the proof of the theorem is complete.
\end{proof}

\end{document}